\documentclass[12pt]{amsart}
\usepackage{times,fullpage}
\usepackage{latexsym}
\usepackage{amssymb}
\usepackage{url}
\usepackage{pictexwd,dcpic}
\usepackage{graphicx}
\usepackage{pinlabel}

\newtheorem{thm}{Theorem}
\newtheorem{prop}{Proposition}
\newtheorem{cor}{Corollary}
\newtheorem{lem}{Lemma}

\theoremstyle{remark}
\newtheorem{rem}{Remark}
\newtheorem{ex}{Example}

\theoremstyle{definition}
\newtheorem{defn}{Definition}

\newcommand{\Q}{\mathbb{Q}}
\newcommand{\R}{\mathbb{R}}
\newcommand{\Z}{\mathbb{Z}}
\DeclareMathOperator{\SL}{SL}

\DeclareMathOperator{\Sol}{Sol}
\DeclareMathOperator{\Nil}{Nil}

\newcommand\id{\operatorname{id}}

\title[On three-manifolds dominated by circle bundles]{On three-manifolds dominated by circle bundles}
\author{D.~Kotschick}
\address{Mathematisches Institut, {\smaller LMU} M\"unchen, Theresienstr.~39, 80333~M\"unchen, Germany}
\email{dieter@member.ams.org}
\author{C.~Neofytidis}
\address{Mathematisches Institut, {\smaller LMU} M\"unchen, Theresienstr.~39, 80333~M\"unchen, Germany}
\email{Christoforos.Neofytidis@mathematik.uni-muenchen.de}
\date{February 22, 2012; \copyright{\ D.~Kotschick and C.~Neofytidis 2012}}
\thanks{We are grateful to J.~Bowden and to P.~Derbez for helpful discussions. This paper was completed
during a visit of the first author to the Institut Mittag-Leffler (Djursholm, Sweden).
The second author is supported by the {\it Deutscher Akademischer Austausch Dienst} (DAAD)}

\begin{document}

\begin{abstract}
We determine which three-manifolds are dominated by products. The result is that a closed, oriented, connected three-manifold is dominated by a product 
if and only if it is finitely covered either by a product or by a connected sum of copies of $S^2 \times S^1$. This characterization can also be formulated in 
terms of Thurston geometries, or in terms of purely algebraic properties of the fundamental group. We also determine which three-manifolds are dominated 
by non-trivial circle bundles, and which three-manifold groups are presentable by products.
\end{abstract}

\maketitle


\section{Introduction}\label{s:introduction}

The study of non-zero degree maps between closed, oriented manifolds has become very active over the last few decades~\cite{CT,Gromov,KL}. 
The existence of a non-zero degree map, $M \longrightarrow N$, defines a transitive relation on the set of homotopy types of closed, oriented manifolds. 
Whenever such a map exists we say that $M$ dominates $N$ and write $M \geq N$. In this case $M$ is at least as complicated as $N$.
For example, the induced maps in rational homology are surjective, thus, in particular, the Betti numbers of $N$ are bounded from above by those of 
$M$. Also, $M \geq N$ implies that the fundamental group of $M$ surjects onto a finite index subgroup of the fundamental group of $N$.

In dimension two, the domination relation coincides with the ordering given by the genus, but in higher dimensions it fails to be an ordering. 
We illustrate this by the following two examples in dimension three. The examples have obvious generalizations to higher dimensions.

\begin{ex}
The two three-manifolds $M=S^3$ and $N=\R P^3$ satisfy $M\geq N$ and $N\geq M$, but fail to be homotopy equivalent.
\end{ex}

\begin{ex}
Let $M$ be a hyperbolic homology three-sphere, and $N=S^1\times S^2$. Then $N$ has larger first Betti number than $M$, and so
$M\ngeq N$. We also have $N\ngeq M$ since the fundamental group of $N$ is infinite cyclic, and so cannot surject onto the fundamental 
group of a closed negatively curved manifold, for example by Preissmann's theorem. Thus $M$ and $N$ are not comparable under the 
domination relation.
\end{ex}

In this paper, we study domination by products for three-manifolds. This is motivated by the work of L\"oh and the first author in~\cite{KL,KL2}, 
where strong restrictions were found for certain manifolds with large universal coverings to be dominated by products. In fact, the results of 
those papers show that three-manifolds dominated by products cannot have hyperbolic or $\Sol^3$-geometry, and must often be prime.
However, in this paper we will not use those earlier results, but follow a more direct approach. This is possible since in dimension three
the only product manifolds are those with a circle factor, and this gives  much stronger constraints than the consideration of 
arbitrary products. The main result we prove here is the following:


\begin{thm}\label{t:topargument}
      A closed, oriented, connected three-manifold $N$ is dominated by a product $\Sigma\times S^1$ if and only if 
      \begin{enumerate}
         \item either $N$ is finitely covered by a product $F\times S^1$, for some aspherical surface $F$, or
         \item $N$ is finitely covered by a connected sum $\#_n(S^2\times S^1)$.
      \end{enumerate}
   \end{thm}
As usual, the empty connected sum corresponding to $n=0$ is $S^3$. 

The proof of Theorem~\ref{t:topargument} falls naturally into two parts. On the one hand, we have to prove that all three-manifolds not 
listed in the statement of the theorem cannot be dominated by products. 
On the other hand, we have to prove that the manifolds listed in the theorem are 
indeed dominated by products. This is obvious for manifolds finitely covered by products, but it is not obvious for the connected sums
occurring in the second statement. Here the proof proceeds by constructing certain maps of non-zero degrees as branched coverings.
This construction, which also has a high-dimensional generalization, is of independent interest. 

Previously, many non-trivial results have been proved about the domination relation in dimension three using a variety of tools different 
from the ones we use here, such as Thurston's geometries, Gromov's simplicial volume, and the Seifert volume. A survey of the state of 
the art at the beginning of the last decade is given in~\cite{Wa}. For more recent results, especially related to the issue of finiteness of 
sets of mapping degrees between three-manifolds, see for example~\cite{DSW} and the papers quoted there. Our proofs here are 
independent of this earlier work, and in fact clarify certain claims made there, cf.~Subsection~\ref{ss:W} below.

It is not immediately obvious to what extent Theorem~\ref{t:topargument} really depends on the assumption that the domains of 
our dominant maps are products, and one could try to replace these products by fibered three-manifolds. For this purpose surface 
bundles over the circle are not interesting, since every three-manifold is dominated by such a bundle by a result of Sakuma~\cite{Sa}.
However, considering non-trivial circle bundles over surfaces we obtain a result parallel to Theorem~\ref{t:topargument}:
 \begin{thm}\label{t:topbundle}
      A closed, oriented, connected three-manifold $N$ is dominated by a non-trivial circle bundle over a surface if and only if 
      \begin{enumerate}
         \item either $N$ is finitely covered by a non-trivial circle bundle over some aspherical surface, or
         \item $N$ is finitely covered by a connected sum $\#_n(S^2\times S^1)$.
      \end{enumerate}
   \end{thm}

In Section~\ref{s:ess} we discuss the notion of rational essentialness in the case of three-manifolds. While this is not logically
necessary for the proofs of our main results, we find it convenient, following~\cite{KL}, to use this concept as an organizing 
principle. In Section~\ref{s:proof1}, respectively Section~\ref{s:proof2}, we then prove Theorems~\ref{t:topargument} and~\ref{t:topbundle}
for rationally essential, respectively inessential, three-manifolds. In Section~\ref{s:algargum} we reformulate these theorems
in terms of Thurston geometries and in purely algebraic terms. Finally, in Section~\ref{s:groups} we determine the three-manifold
groups presentable by products, and in Section~\ref{s:dis} we make some further remarks. These last two sections contain
two new characterizations of (aspherical) Seifert manifolds.


\section{Rational essentialness for three-manifolds}\label{s:ess}

The obstructions for domination by products found in~\cite{KL} are applicable to rationally essential manifolds in the sense of 
the following definition going back to Gromov~\cite{Gromov}:
\begin{defn}
  A closed, oriented, connected $n$-manifold $N$ is called rationally essential if 
  $$
  H_n(c_N)([N])\neq 0\in H_n(B\pi_1(N);\Q) \ , 
  $$
  where $c_N\colon N\longrightarrow B\pi_1(N)$ classifies the universal covering of $N$. 
\end{defn}
For three-manifolds, this definition can be interpreted in terms of the Kneser-Milnor prime decomposition~\cite{M}.
Recall that this says that a closed oriented connected three-manifold $N$ has an essentially unique prime decomposition
$N=N_1\#\cdots\# N_k$ under the connected sum operation. Each prime summand $N_i$ is either aspherical, is $S^1\times S^2$, 
or has finite fundamental group. 
We now have the following:
\begin{thm}\label{t:ess}
For a closed oriented connected three-manifold $N$ the following conditions are equivalent:
\begin{enumerate}
\item[\normalfont{($\Q$ESS)}] $N$ is rationally essential,
\item[\normalfont{(ASPH)}] $N$ has an aspherical summand $N_i$ in its prime decomposition,
\item[\normalfont{(NFREE)}] $N$ is not finitely covered by a connected sum $\#_n(S^2\times S^1)$,
\item[\normalfont{(ENL)}] $N$ is compactly enlargeable,
\item[\normalfont{(NPSC)}] $N$ does not admit a metric of positive scalar curvature.
\end{enumerate}
\end{thm}
The last two items are not relevant to the main results of this paper, so we will only discuss them briefly.
\begin{proof}
A connected sum is rationally essential if and only if at least one of the summands is. Since $S^1\times S^2$ and manifolds 
with finite fundamental group are not rationally essential, this proves the equivalence of ($\Q$ESS) and (ASPH). 

It is obvious that (ASPH) implies (NFREE).
For the converse assume that $N$ contains no aspherical summands in its prime decomposition, i.e.~that $N$ has the form
   \begin{align*}
      N = \underbrace{(S^2\times S^1) \#\cdots\# (S^2\times S^1)}_l \# \underbrace{(S^3/Q_{l+1}) \#\cdots\# (S^3/Q_k)}_{k-l} \ ,
   \end{align*}
where the empty connected sum ($k=0$) denotes the 3-sphere $S^3$. The summands $S^2\times S^1$ have infinite cyclic fundamental groups and the summands $S^3/Q_i$ 
have finite fundamental groups $Q_i$, $l+1\leq i\leq k$. Thus, the fundamental group of $N$ is the free product
   \begin{align*}
      \pi_1(N) = F_l\ast Q_{l+1} \ast\cdots\ast Q_k \ ,
   \end{align*}
where $F_l$ is a free group on $l$ generators. We project this free product to the direct product of the $Q_j$  to obtain the following exact sequence:
   \begin{equation}\label{eq.1}
    1 \longrightarrow \ker(\varphi)
      \longrightarrow \pi_1(N) = F_l \ast Q_{l+1} \ast\cdots\ast Q_k
      \stackrel{\varphi}\longrightarrow Q_{l+1} \times\cdots\times Q_k
      \longrightarrow 1\ .
   \end{equation}
By the Kurosh subgroup theorem, $\ker(\varphi)$ is a free group $F_n$. Since it has finite index in $\pi_1(N)$, we see that $N$ has a finite covering whose fundamental group is free. 
By Kneser's prime decomposition theorem and Grushko's theorem, we deduce that this covering is a connected sum of $n$ copies of $S^2\times S^1$. 
This means that $N$ is finitely covered by a connected sum $\#_n(S^2\times S^1)$, where $n$ is the number of generators of the free group $\ker (\varphi)$ in the exact sequence (\ref{eq.1}).

To see that (ASPH) implies (ENL) it is enough to show that any aspherical three-manifold $N$ is compactly enlargeable. This was proved 
by Gromov and Lawson~\cite[Theorem~6.1]{GL} under the assumptions that $\pi_1(N)$ is residually finite and contains an infinite surface
group. It is now known that all three-manifold groups are residually finite~\cite{H}. (This reference treats only manifolds satisfying Thurston's geometrisation conjecture,
which has now been verified by Perelman~\cite{P1,P2,KLott}.) Furthermore, if $\pi_1(N)$ contains no infinite surface group, then $N$ is atoroidal, and so is hyperbolic 
by Perelman's work~\cite{P1,P2,KLott}. Since hyperbolic manifolds are compactly enlargeable by~\cite[Prop.~3.3]{GL}, we conclude that (ASPH) implies (ENL).

Gromov and Lawson~\cite[Theorem~3.7]{GL} proved that (ENL) implies (NPSC). (Recall that all oriented three-manifolds are spin.)

Finally, (NPSC) implies (ASPH) because $S^1\times S^2$ has positive scalar curvature, and so do all three-manifolds with finite fundamental 
group by Perelman's proof of the Poincar\'e conjecture~\cite{P1,P2,MT}. A connected sum of manifolds with positive scalar curvature also has positive
scalar curvature by the construction of Gromov and Lawson; cf.~\cite[Theorem~5.4]{GL}.
\end{proof}

\begin{rem}
It was proved by Hanke and Schick~\cite{HS} that (ENL) implies ($\Q$ESS) in all dimensions. The converse is not true in dimensions $\geq 4$
by a recent result of Brunnbauer and Hanke~\cite{BH}.
\end{rem}

\section{Rationally essential targets}\label{s:proof1}

In view of Theorem~\ref{t:ess}, the proofs of Theorems~\ref{t:topargument} and~\ref{t:topbundle} split into two cases, depending on whether $N$ 
contains an aspherical summand $N_i$ in its prime decomposition, or not. In this section we deal with the case where an aspherical summand 
does appear.

The first part of Theorem~\ref{t:topargument} corresponds to the following statement:
\begin{prop}\label{p:ess}
A rationally essential closed oriented three-manifold $N$ is dominated by a product if and only if it is finitely covered by a product $F\times S^1$,
with $F$ an aspherical surface.
\end{prop}
\begin{proof}
A manifold finitely covered by a product is of course dominated by that product.

For the converse assume that a product $\Sigma\times S^1$ of a closed, oriented, connected surface $\Sigma$ with the circle dominates 
a closed, oriented, connected rationally essential three-manifold $N$ and let  $f \colon \Sigma\times S^1 \longrightarrow N$ be a map of 
non-zero degree. Then $\Sigma\times S^1$ must be rationally essential, and so $\Sigma$ is of positive genus. 

By replacing $N$ by a finite covering if necessary, we may assume that $f$ is $\pi_1$-surjective.
Let $N=N_1\#\cdots\# N_k$ be the Kneser-Milnor prime decomposition of $N$. Each prime summand $N_i$ is either aspherical, is $S^1\times S^2$, 
or has finite fundamental group; cf.~\cite{M}. By Theorem~\ref{t:ess} the rational essentialness of $N$ is equivalent to the existence
of an aspherical summand $N_i$. Composing $f$
with the degree one map $N\longrightarrow N_i$ collapsing the connected summands other than $N_i$, we obtain a 
dominant map $\Sigma\times S^1\longrightarrow N_i$ between aspherical three-manifolds. This cannot factor through $\Sigma$,
implying that $\pi_1(f)$ must be non-trivial on the central $\Z$-subgroup generated by the $S^1$ factor. But then 
$\pi_1(f)(\Z)$ is a non-trivial central subgroup in $\pi_1(N)$, and so this group is freely indecomposable. 
Thus we may assume that $N$ itself is prime and aspherical, for we can either appeal to Perelman's proof of the Poincar\'e conjecture~\cite{P1,P2,MT} 
to conclude $N=N_i$, or we can argue that the assumption $\Sigma\times S^1\geq N$ depends only on the homotopy type of $N$, which does not change 
if we replace a manifold by its connected sum with a homotopy sphere. 

We have shown that $N$ is aspherical, and that its fundamental group has infinite center. If $N$ is Haken, then it follows from a result of 
Waldhausen~\cite{W} that $N$ is Seifert fibered. In fact, even without the Haken condition, $N$ must be Seifert fibered, by the proof of the Seifert fiber 
space conjecture (stated in~\cite[p.~484]{S} and proved by Casson--Jungreis~\cite{CJ} and Gabai~\cite{Gabai}). Therefore, after lifting 
$f$ to a suitable covering space, we may assume that $N$ is a circle bundle over an aspherical surface. It remains to show that the 
Euler number of this circle bundle is zero. We will prove this in the following lemma, thereby completing the proof of Proposition~\ref{p:ess}.
\end{proof}

\begin{lem}\label{lem}
Let $\pi\colon N\longrightarrow F$ be an oriented circle bundle with non-zero Euler number over a closed aspherical surface. 
Then every continuous map $f\colon\Sigma\times S^1\longrightarrow N$ has degree zero.
\end{lem}
\begin{proof}
Since $N$ is aspherical, we may assume that $\Sigma$ has positive genus. From the discussion above we may assume that
$\pi_1(f)(S^1)$ is an element of infinite order in the center of $\pi_1(N)$.

Since on a Seifert manifold elements of the center of the fundamental group are, up to taking multiples, fibers of Seifert
fibrations, cf.~\cite[p.~92/93]{J} and~\cite{S}, we may assume that $\pi_1(f)(S^1)$ is a multiple of the homotopy class of the fiber in $N$. 
(The fibration of $N$ is unique, cf.~\cite[Thm.~3.8]{S}.)
Thus the composition $\pi\circ f$ kills the homotopy class of the $S^1$-factor in $\Sigma\times S^1$. Since 
$F$ is aspherical, this implies that $\pi\circ f$ is homotopic to a map that factors through the projection 
$\pi_1\colon\Sigma\times S^1\longrightarrow\Sigma$. By the homotopy lifting property of
$\pi\colon N\longrightarrow F$, the homotopy of $\pi\circ f$ can
be lifted to a homotopy of $f$, so we may assume that $\pi\circ f = \bar f \circ\pi_1$ for some continuous map 
$\bar f\colon\Sigma\longrightarrow F$.

Since $\pi\colon N\longrightarrow F$ has non-zero Euler number, $\pi$ induces the zero map on $H_2(N,\Q)$,
and the fundamental class of $F$ is not in the image. As $\pi_1$ is surjective on $H_2$, the equation $\pi\circ f = \bar f \circ\pi_1$
shows that $\deg (\bar f)=0$. Now consider the pullback of $N$ under $\bar f$:
$$
\bar f^*N = \{ (p,x)\in\Sigma\times N \ \vert \ \bar f (p) = \pi (x) \} \ .
$$
The map $f\colon\Sigma\times S^1\longrightarrow N$ factors through $\bar f^*N$ as follows:
\begin{alignat*}{2}
f\colon\Sigma\times S^1 &\longrightarrow \bar f^*N &&\stackrel{\pi_2}{\longrightarrow} N\\
(p,\theta) &\longmapsto (p,f(p,\theta )) &&\longmapsto f(p,\theta ) \ .
\end{alignat*}
For any pullback of an oriented bundle, the degree of the map between total spaces is the same as the degree 
of the map of base spaces under which the bundle is pulled back. In our situation this says that the degree of 
$\pi_2\colon  \bar f^*N\longrightarrow N$ equals the degree of $\bar f$, which vanishes. Thus $f$ factors through 
a degree zero map, and we finally have $\deg (f) =0$.
\end{proof}

The next proposition covers the first part of Theorem~\ref{t:topbundle}.
\begin{prop}\label{p:essbd}
A rationally essential closed oriented three-manifold $N$ is dominated by a non-trivial circle bundle over a surface if and only if it is 
finitely covered by a non-trivial circle bundle over some aspherical surface.
\end{prop}
\begin{proof}
Let $f\colon M\longrightarrow N$ be a map of non-zero degree, with $M$ a non-trivial circle bundle over a surface $\Sigma_g$
of genus $g$. After replacing $N$ by a suitable covering, we may assume that $f$ is $\pi_1$-surjective.
Since $N$ is assumed to be rationally essential, $\pi_1(M)$ must be infinite, and so $g>0$. This means that 
$M$ is aspherical and we have a non-trivial central extension
$$
1\longrightarrow\Z\longrightarrow\pi_1(M)\longrightarrow\pi_1(\Sigma_g)\longrightarrow 1 \ .
$$

The prime decomposition of $N$ contains an aspherical summand $N_i$ by Theorem~\ref{t:ess}. Composing $f$
with the degree one map $N\longrightarrow N_i$ collapsing the connected summands other than $N_i$, we obtain a 
dominant map $M\longrightarrow N_i$ between aspherical three-manifolds. This cannot factor through $\Sigma_g$,
implying that $\pi_1(f)$ must be non-trivial on the central $\Z$-subgroup generated by the circle fibers in $M$. But then 
$\pi_1(f)(\Z)$ is a non-trivial central subgroup in $\pi_1(N)$, and so $N$ is prime and therefore irreducible and aspherical 
itself. As in the proof of Proposition~\ref{p:ess} we conclude that $N$ is Seifert fibered.

After replacing $M$ and $N$ by suitable coverings, we may assume that $N$ is also a circle bundle. It remains to
prove that it has non-trivial Euler class. Now $\pi_1(f)$ sends the element of $\pi_1(M)$ represented by
the circle fibers in $M$ to a non-trivial element of the center of $\pi_1(N)$. This group is torsion-free, so this 
non-trivial element has infinite order. Some multiple of it is the fiber of a Seifert fibration of $N$, cf.~\cite[p.~92/93]{J}.
As mentioned before, we may assume that this Seifert fibration is a circle bundle. Since the fiber in $M$ has finite 
order in homology because the Euler class of $M$ was non-zero, it follows that the circle fiber in $N$, being, up to taking multiples,
the image under $H_1(f)$ of the circle fiber in $M$, also has finite order in homology, and so the Euler class of $N$ 
must be non-zero.
\end{proof}

\section{Rationally inessential targets}\label{s:proof2}

In this section we prove Theorems~\ref{t:topargument} and~\ref{t:topbundle} in the case of rationally inessential manifolds,
i.~e.~those with no aspherical summand in their prime decomposition. The proof is constructive, exhibiting certain dominant
maps as branched coverings.

The second part of Theorem~\ref{t:topargument} corresponds to the following statement:
\begin{prop}\label{p:iness}
Every rationally inessential three-manifold is dominated by a product.
\end{prop}
Since we have shown in the proof of Theorem~\ref{t:ess} that rationally inessential three-manifolds are finitely covered by 
connected sums of copies of $S^1\times S^2$, it suffices to prove the following:
\begin{prop}\label{p:branchcov}
Let $\Sigma_n$ be a closed, oriented surface of genus $n$. For every $n$ the manifold $\Sigma_n\times S^1$ is a $\pi_1$-surjective
branched double covering of $\#_n(S^2\times S^1)$.
\end{prop}
\begin{proof}
\begin{figure}
\labellist
\pinlabel $\stackrel{P}\longrightarrow$ at 300 87
\pinlabel $\cong$ at 590 87
\endlabellist
\centering
\includegraphics[width=12cm]{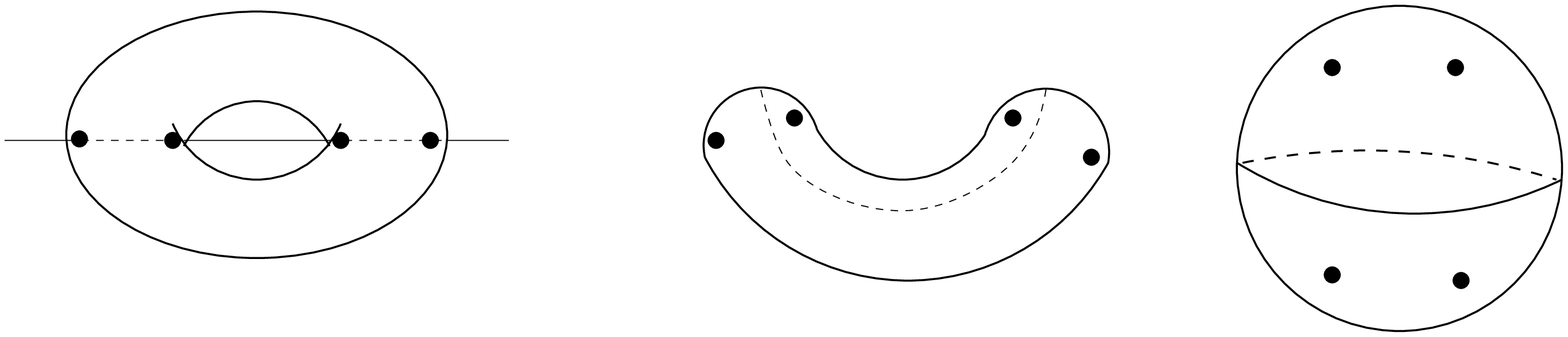}
\caption{\small The branched covering $P \colon T^2 \longrightarrow S^2$.}
\label{f:branchcov1}
\end{figure}

The $2$-torus $T^2$ is a branched double covering of $S^2$ with four branch points.
We denote this branched covering, which is the quotient map for an involution on $T^2$, by $P \colon T^2 \longrightarrow S^2$; see Figure~\ref{f:branchcov1}. (The letter $P$ stands
either for ``pillowcase'', or for the Weierstrass $\mathfrak{p}$-function.)
We multiply $P$ by the identity map on $S^1$ to obtain a branched double covering
   \begin{equation}\label{eq.2}
      P\times \id_{S^1} \colon T^2\times S^1 \longrightarrow S^2\times S^1 \ .
   \end{equation}
 This is the case $n=1$ in the claim. 
 
\begin{figure}
\labellist
\pinlabel {\tiny $D^2=P(A)$} at 251 143
\endlabellist
\centering
\includegraphics[width=8cm]{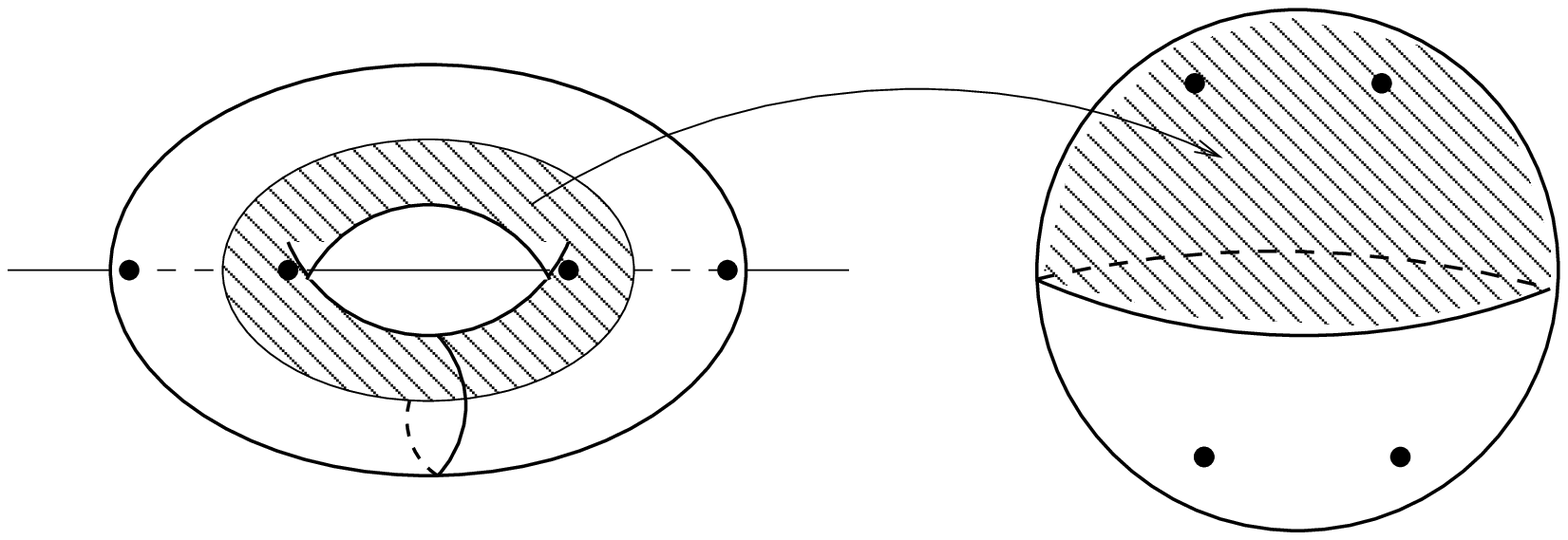}
\caption{\small The preimage of $D^2$, with two branched points, is an annulus in $T^2$.}
\label{f:branchcov2}
\end{figure}

Now let $D^2$ be a $2$-ball in $S^2$ that contains exactly two branch points of $P$ in its interior, as shown in Figure~\ref{f:branchcov2}, 
and let $I$ be an interval in $S^1$. The product $D^2\times I$ is a $3$-ball $D^3$ in $S^2\times S^1$. The preimage of this ball under 
$P\times \id_{S^1}$ is
   \begin{align*}
      (P\times \id_{S^1})^{-1}(D^2\times I) = A\times I \ ,
   \end{align*}
where $A$ is an annulus in $T^2$; see Figure~\ref{f:branchcov2}. We remove this $D^3$ from $S^2\times S^1$ and its preimage from $T^2\times S^1$ 
to obtain a branched double covering
   \begin{equation}\label{eq.3}
      (T^2\times S^1)\setminus (A\times I) \longrightarrow (S^2\times S^1)\setminus (D^2\times I) \ ,
   \end{equation}
where $(T^2\times S^1) \setminus (A\times I) = (T^2 \setminus D^2)\times S^1$. Taking the double of \eqref{eq.3} we obtain a branched double covering
   \begin{equation}\label{eq:n=2}
      \Sigma_2\times S^1 \longrightarrow(S^2\times S^1) \# (S^2\times S^1) \ ,
   \end{equation}
which is $\pi_1$-surjective by construction. This gives the case $n=2$ in the claim.

Finally note that, for arbitrary $n$, the connected sum $\#_n(S^2\times S^1)$ is an $(n-1)$-sheeted unramified covering of $(S^2\times S^1) \# (S^2\times S^1)$. 
Taking the fiber product with~\eqref{eq:n=2}, we obtain the desired $\pi_1$-surjective branched double covering of $\#_n(S^2\times S^1)$ by
$\Sigma_n\times S^1$. This completes the proof.
\end{proof}

Proposition~\ref{p:branchcov} together with Proposition~\ref{p:ess} completes the proof of Theorem~\ref{t:topargument}.
For Theorem~\ref{t:topbundle} we need the following:
\begin{prop}\label{p:inessbundles}
Every rationally inessential three-manifold is dominated by a non-trivial circle bundle over a surface.
\end{prop}
This, together with Proposition~\ref{p:essbd}, completes the proof of Theorem~\ref{t:topbundle}. 
Since every rationally inessential three-manifold is finitely covered by some $\#_n(S^2\times S^1)$ by the 
proof of Theorem~\ref{t:ess}, Proposition~\ref{p:inessbundles} is a consequence of the following statement.
\begin{prop}\label{p:branchcovbundles}
For every $n$ the connected sum $\#_n(S^2\times S^1)$ admits a $\pi_1$-surjective branched double covering 
by a non-trivial circle bundle over a surface.
\end{prop}
\begin{proof}
For $n=0$, the empty connected sum is, by convention, the three-sphere $S^3$, which, via the Hopf fibration, is a non-trivial circle
bundle over $S^2$. Pulling back the Hopf fibration under a branched double cover $S^2\longrightarrow S^2$, we 
obtain the desired double branched cover of $S^3$.

For $n=1$ we prove that the total space $M$ of the circle bundle with Euler number $=1$ over $T^2$ is a  
$\pi_1$-surjective branched double covering of $S^2\times S^1$. Start by considering $M$ as 
the mapping torus of the linear torus diffeomorphism given by the matrix 
$$
\varphi =  \begin{pmatrix} 1 & 1\\ 0 & 1\end{pmatrix} \ ,
$$
and recall that the double branched cover $P\colon T^2\longrightarrow S^2$ in Figure~\ref{f:branchcov1} 
is the quotient map for the involution 
$$
\iota =  \begin{pmatrix} -1 & 0\\ 0 & -1\end{pmatrix} \ .
$$
Since $\iota$ commutes with $\varphi$, it induces a fiber-preserving involution, also denoted $\iota$, of the 
mapping torus $M=M(\varphi)$. The quotient $M/\iota$ is the mapping torus of the diffeomorphism of $T^2/\iota = S^2$
induced by $\varphi$. This diffeomorphism is orientation-preserving, and so $M/\iota = S^2\times S^1$. The projection
$P_M\colon M\longrightarrow M/\iota = S^2\times S^1$ given by the quotient map for $\iota$ is the desired $\pi_1$-surjective 
double branched cover. On every fiber it coincides with $P$.

To deal with the case $n>1$, we revert to thinking of $M$ as a circle bundle over $T^2$, and we fiber sum $n$
copies of this circle bundle to obtain a circle bundle with Euler number $=n$ over $\Sigma_n$. We can perform this 
fiber sum in such a way that the branched double covering maps $P_M$ on the different summands fit together
to give the desired $\pi_1$-surjective branched double covering of $\#_n(S^2\times S^1)$. Recall that the circles 
of the circle fibration of $M$ over $T^2$ are contained in the fibers of the mapping torus projection 
$\pi\colon M=M(\varphi)\longrightarrow S^1$. Pick one such fiber, and thicken it to an annulus $A$ contained in a 
fiber of $\pi$ whose image under $P$ is a disk in $S^2$ containing precisely two branch points of $P$, as shown in
Figure~\ref{f:branchcov2}. A fibered neighbourhood of our circle fiber in $M$ is the 
product of $A$ with an interval in $S^1$, and the image under $P_M$ of this fibered neighbourhood in $S^2\times S^1$ is a 
three-ball $D^3$. Now we can perform the connected sum of two copies of $S^2\times S^1$ along this $D^3$, and 
simultaneously fiber sum two copies of $M$ by removing the fibered neighbourhood and gluing the boundary tori
in a fiber-preserving way that matches up the branch loci. This completes the proof for $n=2$, and the 
general case follows by iterating the construction.
\end{proof}

\section{Geometric and algebraic reformulations}\label{s:algargum}

We now reformulate Theorems~\ref{t:topargument} and~\ref{t:topbundle} and their proofs to obtain equivalent formulations
in terms of Thurston geometries and in terms of purely algebraic properties of fundamental groups.

The following is the geometric reformulation of Theorem~\ref{t:topargument}.
   \begin{thm}\label{t:geoargument}
      A closed, oriented, connected three-manifold $N$ is dominated by a product $\Sigma\times S^1$ if and only if 
      \begin{enumerate}
         \item either $N$ possesses one of the geometries $\R^3$ or $\mathbb{H}^2 \times \R$, or
         \item $N$ is a connected sum of manifolds possessing the geometries $S^2 \times \R$ or $S^3$.
      \end{enumerate}
   \end{thm}
\begin{proof}
Let $N$ be a closed oriented three-manifold dominated by a product $\Sigma\times S^1$. If the prime decomposition of $N$ contains an
aspherical summand, then we have seen in the proof of Theorem~\ref{t:topargument} that $N$ itself is aspherical, and is finitely covered 
by a product $F\times S^1$, with $F$ of positive genus. In addition, $N$ is Seifert fibered since its finite covering $F\times S^1$ is, cf.~\cite{Scott}. Moreover,
$N$ carries the same Thurston geometry as this covering, namely the $\R^3$ geometry if $F$ has genus one, or the $\mathbb{H}^2 \times \R$
geometry if the genus of $F$ is at least $2$. Conversely, every manifold with one of these geometries is indeed finitely covered, and, therefore,
dominated by a product $F\times S^1$, cf.~\cite{S}.

If the prime decomposition of $N$ does not contain an aspherical summand, then each prime summand is either $S^1\times S^2$, with geometry
$S^2\times\R$, or has finite fundamental group, and thus carries the $S^3$ geometry by the work of Perelman~\cite{P1,P2,MT}. 
For all connected sums with only these summands we have proved in the proof of Theorem~\ref{t:topargument} that they are dominated by products.
\end{proof}

Finally, we give an algebraic formulation, in terms of properties of the fundamental group of the target.
   \begin{thm}\label{t:algargument}
      A closed, oriented, connected three-manifold $N$ is dominated by a product $\Sigma\times S^1$ if and only if 
      \begin{enumerate}
         \item either $\pi_1(N)$ is virtually $\pi_1(F)\times\Z$, for some aspherical surface $F$, or
         \item $\pi_1(N)$ is virtually free.
      \end{enumerate}
   \end{thm}
\begin{proof}
If $N$ is a closed oriented three-manifold dominated by a product $\Sigma\times S^1$, and the prime decomposition of $N$ contains an
aspherical summand, then we have seen in the proof of Theorem~\ref{t:topargument} that  $\pi_1(N)$ is virtually $\pi_1(F)\times\Z$. 
Conversely, if $N$ has a finite covering $\bar N$ with fundamental group  $\pi_1(F)\times\Z$, then this covering is prime as its 
fundamental group is freely indecomposable. Since $F$ is not $S^2$, it follows that $\bar N$ is irreducible and aspherical~\cite{M}.
Thus $\bar N$ is homotopy equivalent to $F\times S^1$, proving that $N$ is dominated by a product.

If the prime decomposition of $N$ does not contain an aspherical summand, then we have seen that $\pi_1(N)$ is virtually free. 
Conversely, if $N$ has virtually free fundamental group, then it is finitely covered by a three-manifold with free fundamental group. 
Kneser's prime decomposition theorem and Grushko's theorem imply that this covering must be a connected sum of copies of $S^2\times S^1$, 
where the number of summands is the number of generators of its fundamental group.
\end{proof}

The analogous reformulations can also be carried out for Theorem~\ref{t:topbundle}. The geometric formulation is:
\begin{thm}\label{t:geobundle}
      A closed, oriented, connected three-manifold $N$ is dominated by a non-trivial circle bundle over a surface if and only if 
      \begin{enumerate}
         \item either $N$ possesses one of the geometries $\Nil^3$ or $\widetilde{\SL_2(\R)}$, or
         \item $N$ is a connected sum of manifolds possessing the geometries $S^2 \times \R$ or $S^3$.
      \end{enumerate}
   \end{thm}
\begin{proof}
We only have to prove the equivalence between the first cases of this theorem and of Theorem~\ref{t:topbundle}. In one direction,
if $N$ has one of the geometries $\Nil^3$ or $\widetilde{\SL_2(\R)}$, then it is finitely covered by a non-trivial circle bundle over an 
aspherical surface~\cite{S}. Conversely, if $N$ is  finitely covered by a non-trivial circle bundle over an aspherical surface, then it is a 
Seifert manifold carrying the same Thurston geometry as this finite covering~\cite{Scott}. 
\end{proof}

The algebraic version of Theorem~\ref{t:topbundle} reads as follows.
\begin{thm}\label{t:algbundle}
      A closed, oriented, connected three-manifold $N$ is dominated by a non-trivial circle bundle over a surface if and only if 
      \begin{enumerate}
         \item either $\pi_1(N)$ has a finite index subgroup $\Gamma$ which fits into a central extension
         $$
         1\longrightarrow\Z\longrightarrow\Gamma\longrightarrow\pi_1(F)\longrightarrow 1 
         $$
         with non-zero Euler class for some aspherical surface $F$, or
         \item $\pi_1(N)$ is virtually free.
      \end{enumerate}
   \end{thm}
\begin{proof}
Again we only have to prove the equivalence between the first cases of this theorem and of Theorem~\ref{t:topbundle}. 
In one direction, if $N$ is finitely covered by a non-trivial circle bundle over an aspherical surface, then its 
fundamental group has a finite index subgroup admitting the required central extension. Conversely, if $\pi_1(N)$
has a finite index subgroup $\Gamma$ fitting into such a central extension, then the corresponding finite covering
has to be prime, irreducible and aspherical, and is therefore homotopy equivalent to the total space of the corresponding 
circle bundle over $F$.
\end{proof}

\section{Three-manifold groups presentable by products}\label{s:groups}

As an algebraic counterpart of our topological results about domination by products for three-manifolds we now want to determine
which fundamental groups of three-manifolds are presentable by products. First we recall the definition:
 \begin{defn}\label{repgroupsdef}(\cite{KL})
    An infinite group $\Gamma$ is \emph{presentable by a product} if there is a homomorphism 
    $\varphi \colon \Gamma_1 \times\Gamma_2\longrightarrow\Gamma$ onto a subgroup of finite index, 
    such that both factors $\Gamma_i$ have infinite image $\varphi(\Gamma_i)\subset\Gamma$.
  \end{defn}
Without loss of generality one can replace each $\Gamma_i$ by its image
in $\Gamma$ under the restriction of~$\varphi$, so that one can assume
the factors $\Gamma_i$ to be subgroups of $\Gamma$ and $\varphi$ to be
multiplication in~$\Gamma$. It is obvious that a group with infinite center $C(\Gamma)$ is presentable by a product
-- just take $\Gamma_1=C(\Gamma)$ and $\Gamma_2=\Gamma$.

The property of (not) being presentable by a product is preserved under passage to finite index subgroups. This 
property was introduced in~\cite{KL} and further studied in~\cite{KL2} because, according to~\cite{KL}, it is a property 
that the fundamental groups of rationally essential manifolds dominated by products must have. 

\begin{thm}
For a closed three-manifold $M$ with infinite fundamental group the following three properties are equivalent:
\begin{enumerate}
\item $\pi_1(M)$ is presentable by a product,
\item $\pi_1(M)$ has a finite index subgroup with infinite center,
\item $M$ is a Seifert manifold.
\end{enumerate}
\end{thm}
\begin{proof}
It is clear that (3) implies (2). The converse is the celebrated Seifert fiber space conjecture, the final cases of which were resolved by 
Casson-Jungreis~\cite{CJ} and by Gabai~\cite{Gabai}.

As noted above, it is also clear that (2) implies (1) for any group. We now prove the converse for three-manifold groups.

By~\cite[Cor.~9.2]{KL2} the only non-trivial free product that is presentable by a product is $\Z_2\star\Z_2$, which is virtually
$\Z$ and so satisfies (2). Thus, we may assume that $\pi_1(M)$ is freely indecomposable, and $M$ is prime. If $\pi_1(M)$ 
is not virtually $\Z$, then $M$ is irreducible and aspherical by the sphere theorem, cf.~\cite{M}. In particular, $\pi_1(M)$ is torsion-free.

By~\cite[Prop.~3.2]{KL}, a torsion-free group $\Gamma$ which is presentable by a product has one of the following
properties:
\begin{itemize}
\item either $\Gamma$ has a finite index subgroup with infinite center, or
\item some finite index subgroup splits as a direct product of infinite groups.
\end{itemize}
Applying this to our $\pi_1(M)$, we have to see that the second alternative in fact implies the first. It is a theorem of 
Epstein~\cite{E} that if the fundamental group of a closed three-manifold splits as a direct product of infinite groups, then one 
of the factors has to be infinite cyclic. But then this factor is central in the whole group.
\end{proof}

\section{Final remarks}\label{s:dis}

\subsection{}
The main result of~\cite{KL} was that for rationally essential manifolds, in any dimension, domination by a product 
implies that the fundamental group is presentable by a product. Theorem~\ref{t:topargument} shows that the converse 
is not true already in dimension three:
Seifert manifolds carrying one of the geometries $\Nil^3$ or $\widetilde{\SL_2(\R)}$ are aspherical and have fundamental
groups presentable by products, but are not dominated by products. (These are the only counterexamples to the converse 
in dimension three.)

The geometry $\widetilde{\SL_2(\R)}$ has another interesting feature relevant to our discussion: $\widetilde{\SL_2(\R)}$
is quasi-isometric to $\mathbb{H}^2 \times \R$, cf.~\cite[IV.48]{dlH}. Compact manifolds with the latter geometry are finitely covered by products,
whereas those with the former geometry are not even dominated by products, although the fundamental groups are presentable 
by products in both cases. It was noted in~\cite[Thm.~10.2]{KL2} that presentability by products is not a quasi-isometry invariant
property of finitely generated groups. This, together with the contrast between manifolds with the geometries $\widetilde{\SL_2(\R)}$
and $\mathbb{H}^2 \times \R$, shows that domination by products cannot be detected by coarse methods, neither at the 
level of groups nor at the level of universal coverings of aspherical manifolds.

\subsection{}
One of the standard characterizations of closed Seifert manifolds is through the property of being finitely covered by circle 
bundles.  In the rationally essential case this can be weakened by replacing finite coverings by arbitrary dominant maps:
\begin{cor}
For a rationally essential closed oriented connected three-manifold the following are equivalent:
\begin{enumerate}
\item it is Seifert fibered,
\item it is finitely covered by a circle bundle,
\item it is dominated by a circle bundle.
\end{enumerate}
\end{cor}
\begin{proof}
It is clear that each of these conditions implies the next one. However, (3) implies (2) by Propositions~\ref{p:ess} and~\ref{p:essbd},
and (2) implies (1) by Scott's result in~\cite{Scott}.
\end{proof}

\subsection{}\label{ss:W}
Our discussion in Section~\ref{s:proof1} shows that there are no maps of non-zero degree between trivial and nontrivial circle bundles
over aspherical surfaces. This statement already appeared in the work of Wang twenty years ago, see~\cite[Theorem~2]{Wa1}. 
However, the proof given there is hard to follow. In particular, there is no argument there for the case covered by our Lemma~\ref{lem}.
At the corresponding place in the proof, compare~\cite[p.~153]{Wa1}, in particular equation (III), Wang seems to argue that a 
group that is presentable by a product must itself be a product, which is of course false. The fundamental groups of Seifert
manifolds with non-zero Euler number are presentable by products, but are not virtually products.


\bibliographystyle{amsplain}

\end{document}